\documentclass[11pt]{amsart}

\usepackage{amsfonts}
\usepackage{graphicx}
\usepackage{amscd}

\usepackage{amsthm}
\usepackage{amstext}
\usepackage{amsmath}
\usepackage{amssymb}
\usepackage[mathscr]{eucal}
\usepackage{url}
\usepackage{verbatim}
\usepackage{mathtools}
\usepackage{booktabs}

\usepackage{tikz}
\newcommand*\circled[1]{\tikz[baseline=(char.base)]{
   \node[shape=circle,draw,inner sep=1pt] (char) {#1};}}

\usepackage[official]{eurosym}

\theoremstyle{plain}

\newtheorem{corollary}{Corollary}

\newtheorem{proposition}{Proposition}
\newtheorem{remark}{Remark}

\numberwithin{equation}{section}

\begin{document}

\title[Maximum of a semigroup's complementary]{The maximum of the complementary of a semigroup with restricted conditions} 
\author{A. Linero Bas and D. Nieves Rold\'an}

\begin{abstract}

We consider the set $$\mathcal{A} = \left\{10\cdot a + 11\cdot b \ | \gcd(a,b)=1, a\geq 1, b\geq 2a+1 \right\}.$$ We will prove that $\mathcal{A}$ is unbounded and that there exists a natural number $M\notin \mathcal{A}$  for which 
$$\left\{M+m:m\geq 1,m\in\mathbb N\right\}\subset \mathcal{A}.$$ Indeed, such number is $M = 1674$.
\end{abstract}
\maketitle
Keywords: Semigroup; boundedness; restricted conditions; coin problem; max-type difference equations. \newline
Mathematics Subject Classification: 39A23; 11A99; 20M99.
 
\section{Introduction}

In this paper, we focus on the set $$\mathcal{A} = \left\{10\cdot a + 11\cdot b \ | \gcd(a,b)=1, a\geq 1, b\geq 2a+1 \right\}.$$ In concrete, we will prove that it is unbounded and has bounded gaps. In fact, it is shown that 1674 is the biggest natural number not included in it.

This set appears when analysing the periodic character of the solutions of the max-type difference equation 
\begin{equation}\label{Eq:G4}
	x_{n+4}=\max\{x_{n+3},x_{n+2}, x_{n+1},0\}-x_n.
\end{equation}

In fact, Linero and Nieves, \cite{Linero}, show that the set of periods, $\mathrm{Per}(F_4)$, of Equation (\ref{Eq:G4}) is $\mathrm{Per}(F_4) = \mathcal{A}\cup \{1,8,11\}$. Therefore, a natural question that arises is whether there exists or not a number $M$ such that every integer number greater than $M$ belongs to $\mathcal{A}$. 

Beforehand, it is interesting to mention that if we had not had the restriction $b \geq 2a+1$, the answer would follow directly by the Diophantic Frobenius Problem. In concrete, the problem -also referred as Coin Problem in the literature- consists in finding the greatest number that can not be expressed as a linear combination with positive coefficients of a set $(a_1,\ldots,a_n)$ of natural numbers with $\gcd(a_1,\ldots,a_n) = 1$. In the particular case of two natural numbers, that greatest number is given by $a_1a_2 - a_1 - a_2$ (for more information, the reader is referred to \cite{Alfonsin}). In our case, $a_1=10$ and $a_2 = 11$, the greatest number that cannot be expressed as a linear combination of $10$ and $11$ is $89$. But as we have highlighted, the conditions $b \geq 2a+1$ and $\gcd(a,b)=1$ complicates the problem.

The main section of this work, namely, Section \ref{Unbounded}, starts analysing the prime numbers in $\mathcal{A}$ and the multiples of 11 that belong to the set. Then, the technique developed to prove the unboundedness of $\mathcal{A}$ and to describe in detail which numbers are in the set is based on the division of the set of natural numbers, not multiple of 11, in ten different classes, $$C_m := \{10m + 11k, k \geq 0\},$$ \noindent where we fix the value $m\in \{1,2,\ldots,10\}$. Finally, we collect in Table \ref{table} all the numbers in $\mathcal{A}$.

\section{Unboundedness of $\mathcal{A}$} \label{Unbounded}

Let us denote the set of prime numbers by $\mathcal{P}$. In the following list, we present the first elements $p$ of $\mathcal{P}$, until $401$, and we encircle those admitting a decomposition $p=10\cdot a+11\cdot b$ with $\gcd(a,b)=1$ and $b\geq 2a+1$:

\noindent $2,3,5,7$, \circled{$11$}, $13,17,19,23,29,31,37,41,$ \circled{$43$}, $47, 53, 59, 61, 67, 71, 73, 79, 83,$ $89,$ \circled{$97$}, 
$101, 103,$ \circled{$107$}, \circled{$109$}, $113, 127,$ \circled{$131$}, $137,$ \circled{$139$}, $149,$ \circled{$151$}, $157,$ \circled{$163$}, $167$, \circled{$173$}, $179, 181, 191$, \circled{$193$}, \circled{$197$}, 
$199, 211, 223,$ \circled{$227$}, \circled{$229$}, $233,$ \circled{$239$}, \circled{$241$}, \circled{$251$}, \circled{$257$}, \circled{$263$}, 
\circled{$269$}, \circled{$271$}, $277,$ \circled{$281$}, \circled{$283$}, \circled{$293$}, \circled{$307$}, \circled{$311$}, \circled{$313$}, \circled{$317$}, \circled{$331$}, \circled{$337$}, \circled{$347$}, 
\circled{$349$}, \circled{$353$}, \circled{$359$}, \circled{$367$}, \circled{$373$}, \circled{$379$}, \circled{$383$}, 
\circled{$389$}, \circled{$397$}, \circled{$401$}. 

For instance, $397\in \mathcal{A}$ since $397=297+100=11\cdot 27 + 10\cdot 10$, whereas $277\notin \mathcal{A}$ as the decompositions $277=187+90=11\cdot 17 + 10\cdot 9$ and $277=77+200=11\cdot 7 + 10\cdot 20$ are not allowed, since $b< 2a+1$ in both cases. We observe that we have encircled all the primes greater than $277$. In fact, this observation can be confirmed by the following result. 
\begin{proposition}\label{P:losprimos}
If $p\in\mathcal{P}$, with $p\geq 281$,  then $p\in \mathcal{A}.$
\end{proposition}
\begin{proof}
According to the above table, we can assume that $p\geq 401$. Notice that $p$ can be written in the form $p=(p-10m)+10m$, where $m$ is the smallest positive integer (in fact, unique, with $m\in\{1,2,\ldots,11\}$) holding that $p-10m$ is a multiple of $11$. Then $p=11\left(\frac{p-10m}{11}\right)+10m$. Let $d=\gcd\left(\frac{p-10m}{11},m\right).$ Then, on one hand, $d|p$ due to the fact that $d$ divides $\frac{p-10m}{11}$ and $m$; on the other hand, since $d\leq m$, and $p$ is prime, $p>12$, we deduce that $d=1$. 
Next, in order to definitively prove that $p\in \mathcal{A}$, we need to show that $\frac{p-10m}{11}\geq 2m+1.$ This is equivalent to see that $p\geq 32m+11$, which obviously occurs because $m\leq 11$ and $p\geq 363$.
\end{proof}

We collect some information about the formation of elements of $\mathcal{A}$ once we know that a prime number $p$ belongs to the set. 
\begin{proposition}\label{P:formacionP}
Let $p\in\mathcal{P}\cap \mathcal{A}$, with $p\geq 43.$ Assume that $p=10a+11b$, with $\gcd(a,b)=1$, $b\geq 2a+1$, $a\geq 1$.  Then:
\begin{itemize}
\item[(a)] $pq\in \mathcal{A}$ for all $q\geq 1$ with $\gcd(p,q)=1$ and $aq\geq 12, (b-2a)q\geq 33$. In particular,  $pq\in \mathcal{A}$ for all $q\geq 33$ with $\gcd(p,q)=1$. 
\item[(b)] $p^kq\in \mathcal{A}$ for all $k\geq 2$ and for all $q\geq 1$ with $\gcd(p,q)=1$. 
\end{itemize}
\end{proposition}
\begin{proof}
(a) We have $p=10a+11b$, with $\gcd(a,b)=1$, $b\geq 2a+1$, $a\geq 1$. For all integer $r$, with $aq\geq 1+11r,$ we can write
 \begin{equation}\label{Eq:pyr}
pq=(10a+11b)q=10(aq)+11(bq)=10(aq-11r)+11(bq+10r).
\end{equation}
Take $r=1$, that is, $pq=10(aq-11)+11(bq+10)$. Obviously, $bq+10\geq 2\left(aq-11\right)+1.$  
Let $d_1=\gcd\left(aq-11,bq+10\right)$.  Notice that $d_1|pq$. If $d_1$ contains some divisor $q'$ of $q$, then $q'|(aq-11)$ and $q'|(bq+10)$; then, $q'| 11 $ and $q'| 10 $, so $q'=1$. Consequently, $d_1$ only contains divisors of $p$, and being $p$ prime, we deduce that $d_1 \in \{1,p\}$. 

-- If $d_1=1$, we finish by imposing the additional property of being $aq\geq 12.$ 

-- If $d_1=p$, since $\gcd(p,q)=1$ and $d|p$, we deduce that $d|(aq)$ and $d|(bq)$, so $d|a$ and $d|b$ according to the above observation on the divisors of $q$. Finally, taking into account that $\gcd(a,b)=1$, we conclude that $d=1$, and $pq\in \mathcal{A}$ whenever $aq\geq 1+ 11p$. 
We refine our bound for the values $q$ in the following way. 
we take $r=-1$ in~(\ref{Eq:pyr}) to set $pq=10(aq+11)+11(bq-10).$ 
Notice that we must have $bq-10\geq 2aq +23$, that is, $bq\geq 2aq +33$. Put $d_2=\gcd(aq+11,bq-10)$.  
Similarly to the above study for $d_1$, it is easy to deduce that $d_2\in\{1,p\}$. If $d_2=1$ we finish the proof by requiring 
$aq\geq 12$ and $bq\geq 2aq +33$ (in particular, if $q\geq 33$, we simultaneously get the two inequalities). Otherwise, for $d_1=d_2=p$, we would obtain $aq-11=pu_1$ and  $aq+11=pu_2$ for some positive integers $u_1,u_2$, and then $22=p(u_2-u_1)$, which would imply that $p$ divides $2$ or $11$, contrary to the fact that $p\geq 43$. 

(b) For $k\geq 2$, we write $p^kq=pp^{k-1}q=(10a+11b)p^{k-1}q=10(ap^{k-1}q)+11(bp^{k-1}q)$, or $p^kq=10(ap^{k-1}q-11)+11(bp^{k-1}q+10).$ Denote $a'=ap^{k-1}q-11, b'=bp^{k-1}q+10.$ Of course, $b'\geq 2a'+1,$ as it can be easily verified. Put $d=\gcd(a',b')$. We are going to prove that $d=1$, which will end the proof. 

If $d$ contains some divisor $q'$ of $q$ in its factorization, then $q'|a'$ and $q'|b'$, thus $q'|11$ and $q'|10$. Therefore, $q'=1$ and $d|p^k$. If $p|d$, then $p$ would divide both $a'$ and $b'$, and consequently $p|10$ and $p|11$ 
(notice that $k-1\geq 1$), so $p=1$, a contradiction. We conclude that $d=1$.    
\end{proof}

In the above results we have excluded $p=11 \in \mathcal{A}$. Next, we obtain the elements of $\mathcal{A}$ of the form $11^kq$, $k\geq 1$, and $\gcd(11,q)=1$. 

\begin{proposition}\label{P:para11}
Consider $11$, $a=0, b=1$. It holds:
\begin{itemize}
\item[(a)] $11^k q\in \mathcal{A}$ for all $k\geq 3$ and for all $q\geq 1$ with $\gcd(11,q)=1$.
\item[(b)] $11^2q\in \mathcal{A}$ for all $q\geq 3$ with $\gcd(11,q)=1$. Moreover, $11^2$ and $11^2\cdot 2$ do not belong to $\mathcal{A}.$
\item[(c)] $11q\in \mathcal{A}$ whenever $\gcd(11,q)=1$ and 
$$q\in\{1\}\cup\left(\{q: q\geq 33\} \setminus\{43, 54, 76, 120\}\right).$$
\end{itemize}
\end{proposition}

\begin{proof}
Write $11^kq=11\cdot\left(11^{k-1}q-10\right)+10\cdot(11).$ Whereas $k\geq 3$ or $k=2, q\geq 3$, the inequalities $11^{k-1}q-10>0$ and $11^{k-1}q-10\geq 23$ are correct, and the decomposition is permitted. Then, obviously $\gcd\left(11,11^{k-1}q-10\right)=1$ and $11^kq\in \mathcal{A}$.  It remains two numbers to be analyzed, namely, $11^2, 11^2\cdot 2.$ Since $11^2=110+11$ and $11^2\cdot 2=110 + 132 =220+20$, we deduce that both numbers do not belong to $\mathcal{A}$. This completes Cases (a) and (b).

(c) Concerning the numbers of the form $11q$, we distinguish several cases. If $1\leq q \leq 10$, then the only possible decomposition is given by 
$11q=11q+10\cdot 0$, thus we exclude these values from $\mathcal{A}$, except $11$. 

If $11\leq q \leq 32$, the numbers $11q$ admit the decompositions: $0+11q, 10\cdot 11 + 11(q-10), 10\cdot 22 + 11(q-20)$, each of which is not admissible as an element of $\mathcal{A}$, either because $\gcd(0,q)=q>1$ or because the inequality $b\geq 2a +1$ is violated. Hence, there are not elements of $\mathcal{A}$ in the set $\{11q:11\leq q\leq 32\}$.

If $q\geq 33$,  we claim that all the elements $11q$ belong to $\mathcal{A}$, except for $q\in\{43,54,76,120\}$. As a first step, we show that $11q\in \mathcal{A}$ whenever $q\geq 353$. This is clear, from Cases (a)-(b), when $q$ is a multiple of $11$. If $q$ and $11$ are  coprime, write $11q=10\cdot(121) + 11\cdot (q-110)$. Notice that $q-110\geq 2\cdot 121 +1$, since $q\geq 353$. Also, $\gcd(121,q-110) = 1$, due to the fact that $q$ and $11$ are coprime numbers.  

Next, we focus on the range $33\leq q\leq 352$. Now, it makes sense to write $11q=10\cdot 11 + 11\cdot (q-10)$ since $q-10\geq 23.$ If $q-10\neq \dot{11}$, then $\gcd(11,q-10)=1$ and we finish. Then, it remains to analyze the case in which $q-10=11m$, $33\leq q \leq 352$.  This gives us the following set of candidates to be elements of $\mathcal{A}$ of type $q=10+11m$, 
\begin{eqnarray*}
C_q^0:&=&\big\{43, 54, 65, 76, 87, 98, 109, 120, 131, 142, 153, 164, 175, 186, 197,\\
 & & 208, 219, 230, 241, 252, 263, 274, 285, 296, 307, 318, 329, 340, 351\big\}.
\end{eqnarray*}
For these elements $q\in C_q^0$, except $43, 54,$ we write $11q=10\cdot(22)+ 11\cdot(q-20)$ (necessarily, $q\geq 65$ if we try to obtain an element of $\mathcal{A}$), and we consider $d_0=\gcd(22,q-20)$. If $q$ is odd, then $q-20$ so is, and $d_0\in\{1,11\}$ in this situation. If $11|d_0$, then $11|(q-20)$ or $d_0|(11m-10)$; as a consequence, $d_0|10$, which is impossible. Hence, $11q$ is an admissible number if $q$ is odd, $65\leq q\leq 353$. So, we discard odd numbers $q$ in $C_q^0$, except $43$, and the set is reduced to $C_q^1:=\big\{43, 54, 76,  98,  120,  142,  164,  186, 
 208,   230,   252,  274,  296,   318,  340\big\}.$

Next, consider the even elements of $C_q^1$, of type $54+22j, 0\leq j\leq 13,$ and write $11q=10\cdot (33) + 11\cdot(q-30)=10\cdot (33) + 11\cdot(24+22j),\, j=0,1,\ldots,13$
(as we are forced to obey the inequality $q-30\geq 2\cdot 33 +1$, we consider $q\geq 97$). 
If $d_1=\gcd(33,24+22j)$, since $24+22j=11(2+2j)+2$ and $33=3\cdot 11$, we deduce that $d_1\in\{1,3\}.$ Therefore, if 
$24+22j$ is not a multiple of $3$, we find that $d_1=1$ and the corresponding value $11q$ belongs to $\mathcal{A}.$
 This happens whenever $j\notin\{0,3,6,9,12\}$. Subjected to the constraint $q\geq 97$,  the above study allows us to say that $11q$ is an element of $\mathcal{A}$, 
except, maybe, if $q$ is confined in the new subset $C_q^2:=\big\{43, 54, 76,   120,   186, 
  252,   318 \big\}.$ 
  
Now, we consider the decomposition $11q=10\cdot (55) + 11\cdot(q-50)$, where it is necessary to assume that $q\geq 161$, so it applies for the numbers 
$186, 252, 318$. Since $q-50$ is then $136, 202, 268$, respectively, we find that $\gcd(55,q-50)=1$ for these numbers, and they belong to $\mathcal{A}$. By this analysis, $C_q^2$ is reduced to $C_q^3:=\big\{43, 54, 76,   120 \big\}.$ For these numbers, it is easy to check that neither of the possible combinations $10a+11b$ satisfy simultaneously the two conditions $\gcd(a,b)=1$ and $b\geq 2a+1$: 
\begin{eqnarray*}
11\cdot 43= 473 &=&10\cdot 11 + 11\cdot 33 = 10\cdot 22  + 11\cdot 23\\
&=&10\cdot 33  + 11\cdot 13=10\cdot 44  + 11\cdot 3, \\
 11\cdot 54 = 594 &=& 10\cdot 11 + 11\cdot 44 = 10\cdot 22  + 11\cdot 34\\&=&10\cdot 33  + 11\cdot 24=10\cdot 44  + 11\cdot 14=10\cdot 55  + 11\cdot 4, \\
 11\cdot 76 = 836 &=& 10\cdot 11 + 11\cdot 66 = 10\cdot 22  + 11\cdot 56 = 10\cdot 33  + 11\cdot 46\\
&=&10\cdot 44  + 11\cdot 36=10\cdot 55  + 11\cdot 26=10\cdot 66  + 11\cdot 16 \\
&=&10\cdot 77  + 11\cdot 6, \\
11\cdot 120 = 1320 &=& 10\cdot 11 + 11\cdot 110 = 10\cdot 22  + 11\cdot 100 = 10\cdot 33  + 11\cdot 90\\
&=&10\cdot 44  + 11\cdot 80=10\cdot 55  + 11\cdot 70=10\cdot 66  + 11\cdot 60\\
&=&10\cdot 77  + 11\cdot 50 = 10\cdot 88  + 11\cdot 40=10\cdot 99  + 11\cdot 30 \\
&=&10\cdot 110  + 11\cdot 20=10\cdot 121  + 11\cdot 10.
\end{eqnarray*}

\end{proof}

\begin{corollary}\label{C:cota11} It holds that $1320 = \max\{11\cdot n: n\geq 1, 11\cdot n\notin \mathcal{A}\} $.
\end{corollary}
Once we have described the main properties concerning prime numbers as well as a detailed study on which multiples of $11$ belong to $\mathcal{A},$ 
we are now interested in proving that the set $\mathcal{NA}:=\mathbb N\setminus \mathcal{A}$ is bounded. In fact our objective is to calculate the maximum of 
$\mathcal{NA}$.

The strategy consists in dividing the set of natural numbers, not multiple of $11$, in ten different classes, $\mathcal{C}_m:=\left\{10m + 11k, \, k\geq 0\right\}$, where we fix the value $m\in\{1,2,\ldots,10\}$. For each class $\mathcal{C}_m$, we will show that $\mathcal{NA}\cap\mathcal{C}_m$ is bounded, and from the inspection of each subset $\mathcal{NA}\cap\mathcal{C}_m$ we will deduce the maximum of $\mathcal{NA}.$ To develop the study of those classes, we use the following basic fact:
\begin{quote}\label{q:Fact}
Given a natural number $N$, $N\neq \dot{11}$, there exists a unique $m\in\{1,2,\ldots,10\}$ such that $N-10m=\dot{11}.$
\end{quote} 

\subsection{The class $\mathcal{C}_1$} Notice that a number $n\in\mathcal{C}_1$ can be written as $n=10+11b$, where $a=1$ and $\gcd(b,1)=1$. To assure that $n$ is in $\mathcal{A}$, it is necessary to require $b\geq 3$. Therefore, $n\in \mathcal{A}$ for all $n\geq 43$. Additionally, $n=10, 21, 32$ belong to $\mathcal{NA}$ and 
$N_1:=\max\left\{ \mathcal{NA}\cap\mathcal{C}_1\right\}=32.$ 

\subsection{The class $\mathcal{C}_2$} Let $n\in\mathcal{C}_2$, $n=20+11b,\, b\geq 0.$ Here, we take $a=2$.

$\bullet$ If $b$ is odd, with $b\geq 5$, then $n\in \mathcal{A}$  since $\gcd(b,a)=1$. On the other hand, for $b=1$ and $b=3$, we find that $31$ and $53$ belong to $\mathcal{NA}.$ 

$\bullet$ Assume that $b$ is even, $b=2j$, with $j\geq 1$ (if $j=0$ we have $n=20\in\mathcal{NA}$). Since $\gcd(a,b)=2$, we need to manage another decomposition, $n=130+11(b-10)$, that is, $n=10\cdot 13 + 11\cdot(2j-10).$ This is possible if $2j-10\geq 27$, so $j\geq 19$. When $j\leq 18$, the numbers $n=130+11(2j-10)$ belong to $\mathcal{NA}.$ 

We stress that $438,460,482,504,526 \in \mathcal{A}$ as we have the decompositions 
\begin{eqnarray*}
438&=&10\cdot 13 +11\cdot 28,\\ 
460&=&10 \cdot 13 +11 \cdot 30, \\ 
482&=&10\cdot 13 +11 \cdot 32,\\
504&=& 10 \cdot 13 +11 \cdot 34,\\
526&=&10 \cdot 13 +11 \cdot 36,
\end{eqnarray*}
but, on the contrary, $416\in\mathcal{NA}$ since the number admits the following decompositions
$$416=10\cdot 2+ 11\cdot 36=10\cdot 13+ 11\cdot 26=10\cdot 24+ 11\cdot 16=10\cdot 35+ 11\cdot 6,$$
and neither verifies simultaneously the two conditions $\gcd(a',b')=1, b'\geq 2a'+1$ on the coefficients of $10a'+11b'$. 

We continue the analysis of this subcase, $n=10\cdot 13 + 11\cdot(2j-10),\,$ $j\geq 19$. If $2j-10$ is not a multiple of $13$, or $j-5\neq \dot{13}$, we ensure that $n$ belongs to $\mathcal{A}$. Instead, if $j=5+13k$, with $k\geq 2$, then $n=10\cdot 13+ 11\cdot (26k)$. Notice that $26|n$, with $n=26\left(5+11k\right).$ 
By descending in the decomposition, we can write $n=10\cdot 35 + 11\cdot(26k-20)$. 

Then, if $26k-20$ is not divisible neither by $5$ nor $7$, we finish once we guarantee that $26k-20\geq 71$, which can be accomplished when $k\geq 4$. In this case, the remaining cases $k=2,3$ provide the numbers $n=702,988$ both of them belonging to $\mathcal{NA}$, as can be easily checked by hand. 

If, on the contrary, $26k-20$ contains $5$ or $7$ in its factor decomposition, with $k\geq 5$, as a first observation we mention that $n$ will be divisible by $26\cdot 5$ or $26\cdot 7$. Secondly, realize that in the sequence $\left\{x_r:r\geq 0\right\}=\left\{13+11r:r\geq 0\right\}=\left\{13,24,35,46,57,68,79,\ldots\right\}$ we find 
$x_{36}=409$, a prime number in $\mathcal{A}$. If in the decompositions 
\begin{equation}\label{Eq:paraCaso2}
n=10\cdot (13+11r) +11(26k-10r)
\end{equation} 
we are not able to achieve $\gcd(13+11r,26k-10r)=1$ for $r=2,\ldots,36$, at least we know that $n$ can be divided by $26$, $5$ or $7$, $79$ and $409$, so $n=409\cdot t$, with $t\geq 26\cdot 79$. By Proposition~\ref{P:formacionP}, this implies that $n\in \mathcal{A}$. In order to carry out this reasoning, for $r=36$ it is necessary that in (\ref{Eq:paraCaso2}) either $26k-360\geq 2\cdot (409)+1$ if $\gcd(409,26k-360)=1$ or $26k-360\geq 1$, otherwise; in both cases, it suffices to take $k\geq 46$. For the remaining values, $5\leq k\leq 45$, let consider $r=6$, $n=10\cdot 79 + 11\cdot(26k-60)$. Notice that $\gcd(79,26k-60)\neq 1$ if and only if $k=57+79s$ with $s\in \mathbb{N}$. Also, $26k-60 \geq 2\cdot 79 + 1$ is equivalent to $k\geq 9$. Thus, for $9\leq k \leq 45$, the decomposition will be acceptable and $n\in \mathcal{A}$. Finally, for $5\leq k \leq 8$, as we assumed that $26k-20$ was divisible by 5 or 7, the only possible value is $k=5$. So, $n = 1560$ and neither of its possible decompositions verify simultaneously the two required conditions $\gcd(a',b')=1$ and $b'\geq 2a'+1$.
Recall that we had assumed that $26k-20$ was divisible by $5$ or $7$. This refines the analysis to the values $k=5,10,15,20,25,30,35, 40, 45$ (divisibility of $13k-10$ by $5$) and $k= 11, 18, 25, 32, 39$ (divisibility of $13k-10$ by $7$), when we replace them in $n=130+286k$.  With the help of a personal computer and a rather rudimentary algorithm, we obtain the values 
$$1560,2990,4420,5850,7280,8710,10140,11570,13000$$ and $$3276,5278,7280,9282,11284,$$ respectively, and we check that, among them, $1560$ is the unique element in $\mathcal{NA}$. It is easy to see that neither of the its possible decompositions verify simultaneously the two required conditions $\gcd(a',b')=1$ and $b'\geq 2a'+1$. For the rest of numbers we find the following admissible pairs $[a',b']$:
\[
[57,220],[57,350],[79,460],[57,610],[57,740],[79,850],[57,1000],[57,1130]
\]
(for $2990,4420,5850,7280,8710,10140,11570,13000$), and 
\[
[79,226],[57,428],[57,610],[79,772],[57,974]
\]
(for $3276,5278,7280,9282,11284$).

This concludes the case $b$ even in $\mathcal{C}_2$. From our analysis, we highlight that 
$$N_2:=\max\left\{\mathcal{NA}\cap \mathcal{C}_2\right\}=1560.$$

 \subsection{The class $\mathcal{C}_3$} We find the maximum of $\mathcal{NA}$ in 
$\mathcal{C}_3=\left\{30 + 11k, \, k\geq 0\right\}.$ Let $n\in\mathcal{C}_3$, $n=30+11b,\, b\geq 0.$ First, notice that $n=30, 41, 52, 63, 74, 85, 96,$ are the values corresponding to $b=0,\ldots,6,$ respectively, and it is simple to see that they belong to $\mathcal{NA}$. Hence, we start with $b\geq 7$ and $a=3$.

$\bullet$ If $b$ is not a multiple of $3$ with $b\geq 7$, we find that $\gcd(b,3)=1$ and all the elements $30+11b$, with $b\neq\dot{3}$ are in $\mathcal{A}.$

$\bullet$ Assume that $b=3j$, with $j\geq 3$. Since $\gcd(3,3j)=3$, we employ the decomposition $n=10\cdot 14 + 11\cdot (3j-10)$. It has sense if we suppose that $j\geq 4$ and, in order to be $n$ in $\mathcal{A}$, we must impose $3j-10\geq 2\cdot 14 +1,$ that is, $j\geq 13$. For the remaining values, $3\leq j\leq 12$, we find the numbers $n\in\left\{129,162,195,228,261,294,327,360,393,426\right\}$ which, as can be checked, belong to $\mathcal{NA}$. (for instance, neither of the decompositions $426=30+396=10\cdot 3 + 11\cdot 36 = 10\cdot 14 + 11\cdot 26 =10\cdot 25 + 11\cdot 16 =10\cdot 36 + 11\cdot 6$) is admissible in order to $426$ belong to $\mathcal{A}$). 

If $\gcd(14,3j-10)=1$, for $j\geq 13$ we guarantee that $n\in \mathcal{A}$. This holds	whenever $j=13$ and for the odd numbers $j\geq 15$ such that $j\neq 15+14u, u\geq 0$, namely $$j\in\{17,19,21,23,25,27, 31,33,35,37,39, 41,\ldots\}.$$
But if $j\geq 14$ is even or $j=15+14k,$ with $k\geq 0$, we obtain that $\gcd(14,3j-10)>1$ ($n$ can be divided either by $2$ or by $7$) and we need to employ another decomposition of $n$. If we vary $a'$ in $10a'+11b'$ in an ascendant way, we obtain the successive decompositions 
\begin{equation}\label{Eq:paraCaso3}
n=10\cdot (14+11r) +11\cdot (3j-10- 10r), \, r\geq 0.
\end{equation} 
For $r=1$, we have $n=10\cdot 25 +11(3j-20)$. In order to obtain elements of $\mathcal{A}$, it must be $3j-20\geq 51$, to wit $j\geq 24$.  Since we know that $j=2k$ with $k\geq 7$ or $j=15+14k,$ with $k\geq 0$, the restriction $j\geq 24$ enforces us to check the values $j=14,15,16,18,20,22$. For them, their respective numbers $n$ are $492, 525, 558, 624, 690, 756$, and it can be verified that all of them are elements of $\mathcal{NA}$. Since we are interested in the location of the maximum of $\mathcal{NA}$, for the convenience of the reader we present the decompositions of $756$ which assure us that this number is not in $\mathcal{A}$:
{\small{\begin{eqnarray*}
756&=&10\cdot 3 + 11\cdot 66=10\cdot 14+ 11\cdot 56= 10\cdot 25 + 11\cdot 46 = 10\cdot 36 + 11\cdot 36\\
&=& 10\cdot 47 + 11\cdot 26=10\cdot 58 + 11\cdot 16 = 10\cdot 69 + 11\cdot 6.
\end{eqnarray*}}}
Now, once we restrict ourselves to $j\geq 24$, we note that if $3j-20$ is not a multiple of $5$ then $n=30+33j$ is an element of $\mathcal{A}$. Since either $j=2k$ ($k\geq 12$)  or $j=15+14k$ (with $k\geq 1$), $3j-20$ is  a multiple of $5$ whenever either $k=5w$ for $w\geq 3$ or $j=15+70w$ for $w\geq 1$, and in these cases is necessary to use a new decomposition. 

We then pass to use $r=3$, i.e, $n=10\cdot 47 +11(3j-40)$. This implicitly means that $3j-40\geq 2\cdot 47 +1,$ so it must be $j\geq 45$. In that case, we need to analyze independently the remaining lest values of $j$, namely $j=30,40$ (recall that we have in this point that $j=10w, w\geq 3$ or $j=15+70w, w\geq 1$, and $n=30+33j$), which provide the numbers $n=1020, 1350$, respectively. Both of them are elements of $\mathcal{NA}$. 

(for instance 
\begin{eqnarray*}
1350&=&10\cdot 3 + 11\cdot 120=10\cdot 14+ 11\cdot 110= 10\cdot 25 + 11\cdot 100\\
&=&10\cdot 36 + 11\cdot 90 = 10\cdot 47 + 11\cdot 80=10\cdot 58 + 11\cdot 70= \cdots,
\end{eqnarray*}
and neither of them satisfy the conditions to be included in $\mathcal{A}$). 
In the sequel, we consider $j=10w$ with $w\geq 5$, or $j=15+70w, w\geq 1$.

Since the case $r=3$  has associate the value $g:=\gcd(47, 3j-40)\in\{1,47\}$, if $3j-40$ is not a multiple of $47$, we will include $n=30+33j$ in $\mathcal{A}$. Otherwise, $3j-40$ is a multiple of $47$ whenever either $30w-40, w\geq 5$ is multiple of $47$ or $3(15+70w)-40, w\geq 1$ so is; that is, whenever either $3w-4=\dot{47}, w\geq 5$ or $1+42w=\dot{47}, w\geq 1$. In the first case, $w=17+47u$, $u\geq 0$, and in the second one, $w=19+47u, u\geq 0$. Hence, either $j=10(17+47u)=170+470u, u\geq 0$, or $j=15+70(19+47u)=1345+3290u, u\geq 0$. For these values of $j$, we must use a new $n=10\cdot (14+11r) +11(3j-10-10r)$. 

In the sequence $\left\{x^{(3)}_r:r\geq 0\right\}=\left\{14+11r:r\geq 0\right\}=$ $\left\{14,25,36,47, \ldots\right\}$ we find 
$x^{(3)}_{27}=311$, a prime number in $\mathcal{A}$. From this fact, if in the decompositions 
$n= 10\cdot (14+11r) + 11(3j-10- 10r)$ we are not able to attain a value $r=1,2,3,\ldots, 27$ such that $\gcd(14+11r,3j-10-10r)=1$, at least we know that $n$ can be divided by $3$, $2$ or $7$, $47$ and $311$, so $n=311\cdot t$, with $t\geq 6\cdot 47$. Then, Proposition~\ref{P:formacionP} gives $n\in \mathcal{A}$. In order to be able to use this argument, for $r=27$ it is necessary that in (\ref{Eq:paraCaso3}) either $3j-10 -10\cdot 27\geq 2\cdot (311)+1$ if $\gcd(311,3j-10-10\cdot 27)=1$ or $3j -10 - 10\cdot 27\geq 1$, otherwise; in both cases, it suffices to take $j\geq 301$. Taking into account that $j=170+470u, u\geq 0$, or $j=1345+3290u, u\geq 0$, this choice is always suitable, except for the value of $j=170$. Here, we find the number $n=30+33\cdot 170= 5640$ which, in fact, also belongs to $\mathcal{A}$ since $5640=10\cdot 91 +11\cdot 430$, with $\gcd(91,430)=1$ and $430\geq 184$.  
This ends the case $b=3j$  in $\mathcal{C}_2$. From our analysis, it is worth mentioning that $$N_3:=\max\left\{\mathcal{NA}\cap \mathcal{C}_3\right\}=1350.$$
	
\subsection{The class $\mathcal{C}_4$} We want to discuss if there are elements of $\mathcal{NA}$ in 
$\mathcal{C}_4=\left\{40 + 11k, \, k\geq 0\right\},$ and to find the maximum of 
$\mathcal{NA}\cap \mathcal{C}_4.$ Let $n\in\mathcal{C}_4$, $n=40+11b,\, b\geq 0.$

The first values of $\mathcal{C}_4$ are $n=40, 51, 62, 73, 84, 95, 106, 117, 128,$ corresponding to $b=0,\ldots,8$ respectively, and it is easily seen that all of them belong to $\mathcal{NA}$. Hence, we start our analysis with $b\geq 9$. Realize that $a=4$.

$\bullet$ If $b$ is odd, $b\geq 9$, then automatically $n=40+11b\in \mathcal{A}$ as the conditions $\gcd(a,b)=1$ and $b\geq 2a+1$ are guaranteed.

$\bullet$ If $b$ is even, $b=2j$, $j\geq 5$, we need to change the decomposition. Firstly, we consider $n=10\cdot 15 + 11\cdot (2j-10)$. Since it must be $2j-10\geq 31$, we need $j\geq 21$. Thus, if $5\leq j\leq 20$, all the numbers $40+22j$ i.e., $\{150,172,194,216,238,260,282,304,326,348,370,392,414,436,458,480\}$ 
belong to $\mathcal{NA}$, (for instance, $480=10\cdot 4 +11\cdot 40=10\cdot 15 +11\cdot 30=10\cdot 26 +11\cdot 20=\cdots$). 

From here, we will assume that $j\geq 21$. If $2j-10$ is neither a multiple of $3$ nor a multiple of $5$, the number $n$ will be in $\mathcal{A}$. Otherwise, either $j=5k, k\geq 5$ if $5|(2j-10)$, or $j=8+3k, \ k\geq 5$, if $3|(2j-10)$. Similarly to the studies developed for $\mathcal{C}_2$ and $\mathcal{C}_3$, we can write $n$ as follows:
\begin{equation}\label{Eq:paraCaso4}
n=10\cdot (15+11r) +11(2j-10-10r), \, r\geq 0.
\end{equation} 
For $r=2$, $n=10\cdot 37+11\cdot (2j-30)$ will be in $\mathcal{A}$ if $2j-30\geq 75$, that is, if $ j\geq 53$. As we consider that $j=5k, k\geq 5$  or $j=8+3k, k\geq 5$, the decomposition cannot be performed if $j=25,30,35,40,45, 50$ or for $j=23, 26, 29, 32, 35, 38, 41, 44, 47, 50$. For those values, their corresponding $n$ are elements of $\mathcal{NA}$. 

we obtain the numbers $\{590,$ $700,$ $810,$ $920,$ $1030,$ $1140\}$ and $ \{612,678,$ $744,$ $810,$ $876,$ $942,$ $ 1008,$ $1074,$ $1140\},$ respectively. All those numbers are elements of $\mathcal{NA}$. 

Until now, he have detected the values of $\mathcal{NA}$ having the form $40+22j$, when $j<53$. Now, 
if $j\geq 53$ and $2j-30$ is coprime with $37$, we finish: every number will belong to $\mathcal{A}$.  Otherwise, if $37|(2j-30)$, then either 
$37|(10k-30)$, with $k\geq 11$ (when $j=5k$) or $37|(6k-14)$, with $k\geq 15$ (when $j=8+3k$); in the first case, $37|(k-3)$, so $k=40+37u,\, u\geq 0$, and in the second one, $37|(3k-7)$, so $k=27+37u,\, u\geq 0$. Then our analysis will continue by considering $j\geq 53$ and $j=5k=200+185u, \geq 0,$ or $j=8+3k=89+111u, \geq 0.$ 
The sequence $\left\{x^{(4)}_r:r\geq 0\right\}=\left\{15+11r:r\geq 0\right\}=\left\{15,26,37,48,59,\ldots\right\}$ contains at least a prime number 
included in $\mathcal{A}$, namely  $x^{(4)}_{32}=367$. Consequently, if in the decompositions (\ref{Eq:paraCaso4}) we cannot find a value $r=1,\ldots, 32$ such that $\gcd(15+11r,2j-10-10r)=1$, at least we know that $n$ can be divided by $37$ and $59$, and $367$, so $n=367\cdot t$, with $t\geq 37\cdot 59$. Then, Proposition~\ref{P:formacionP} implies that $n\in \mathcal{A}$. In order to develop rightly this reasoning, for $r=32$ it is necessary that in (\ref{Eq:paraCaso4}) either $2j-10 -10\cdot 32\geq 2\cdot (367)+1$ if $\gcd(367,2j-10-10\cdot 32)=1$ or $2j -10 - 10\cdot 32\geq 1$, otherwise; in both cases, it is sufficient to take $j\geq 533$. Since $j= 200+185u, u\geq 0$, 
or $j= 89+111u, u\geq 0$, the choice $j\geq533$ is appropriate, except for the values $j=89, 200, 311, 385, 422$. For them, the corresponding numbers $n=40+22j$ are 
$1998, 4440, 6882, 8510, 9324$, belonging to $\mathcal{A}$. This ends the case $b=2j$  in $\mathcal{C}_4$ and we stress that $$N_4:=\max\left\{\mathcal{NA}\cap \mathcal{C}_4\right\}= 1140.$$

\subsection{The class $\mathcal{C}_5$} We study the set $\mathcal{NA}$ in $\mathcal{C}_5=\left\{50 + 11k, \, k\geq 0\right\}$. Let $n\in\mathcal{C}_5$, 
$n=50+11b,\, b\geq 0.$
The first values of $\mathcal{C}_5$ are $n=50,$ $ 61,$ $ 72,$ $ 83,$ $ 94,$ $ 105,$ $ 116,$ $ 127,$ $ 138,$ $149,$ $160,$ corresponding to $b=0,1,\ldots,10$ respectively, and it is direct to check that all of them belong to $\mathcal{NA}$. Hence, we start our analysis with $b\geq 11$. Realize that, at the beginning, $a=5$.

$\bullet$ If $b$ is not a multiple of $5$, then $\gcd(5,b)=1$. Additionally, as we have supposed that $b\geq 11$, the condition $b\geq 2a+1$ is also fulfilled, and therefore every number $n=50+11b$, with $b\geq 11$ and $b\neq \dot{5}$, are included in $\mathcal{A}$. 

$\bullet$ From now on, we assume that $b=5j,$ with $j\geq 3$. Being unsuccessful the decomposition $n=10\cdot 5+11\cdot(5j)$ in the search of elements of $\mathcal{A}$, 
we use the new decomposition $n=10\cdot 16 + 11\cdot (5j-10).$

Notice that if $j$ is odd, then $\gcd(16,5j-10)=1$. Also, we need $5j-10\geq 33$ or $j\geq 9$. Then, $n=50+55j\in \mathcal{A}$ for all $j\geq 9$, $j$ odd. We exclude the cases $j=3,5,7$. For them, the corresponding $n$ are $215, 325, 435$, all of them in $\mathcal{NA}$, as can be easily verified. 

Now we assume that $j$ is even, $j\geq 4$, thus $j=2k$ and $b=10k$, $k\geq 2$. We apply the decomposition $n=10\cdot 27 + 11\cdot (5j-20).$ If $3$ is not a divisor of $5j-20$, or equivalently, if $j-4\neq\dot{3}$, then $n$ will be in $\mathcal{A}$ whenever $j\geq 15$. With respect to the remaining even values $j=6,8,12,14$, with $j-4\neq \dot{3}$, we find $n=50+55j\in\{380, 490, 710, 820\}$, all of them in $\mathcal{NA}$, (for instance, $820=10\cdot 5 +11\cdot 70=10\cdot 16 +11\cdot 60=10\cdot 27 +11\cdot 50=10\cdot 38 +11\cdot 40=10\cdot 49 +11\cdot 30=\cdots$).  

Otherwise, assume $j=2k, k\geq 2$, $j-4=\dot{3}$, hence $k=2+3u, \,$ and $j=4+6u, \, u\geq 0$. Then $n=10\cdot 27 + 11\cdot (30u),$ with $u\geq 0$. Next, consider the decomposition $n=10\cdot 49 + 11\cdot (30u-20).$ If $3u-2$ is not a multiple of $7$, we could ensure that $n\in \mathcal{A}$ if $30u-20\geq 99$, that is, if $u\geq 4$. Hence, if $30u-20=5j-40$ is not a multiple of $7$, with $u\geq 4$, then $n$ is in $\mathcal{A}$. In this analysis we have excluded the values $u=0,1,2$ (for which $3u-2$ is not a multiple of $7$), or the corresponding values $n=50 + 55(4+6u)\in\{270,600,930\}$, which are in $\mathcal{NA}$, (for instance, $930=10\cdot 5 + 11\cdot 80=10\cdot 16 + 11\cdot 70=10\cdot 27 + 11\cdot 60=10\cdot 38 + 11\cdot 50=10\cdot 49 + 11\cdot 40=\cdots$). 

If $3u-2$ is a multiple of $7$, with $u\geq 3$, we have $u=3+7\ell,$ with $\ell\geq 0$, and consequently $j=4+6u=22+42\ell$, $\ell\geq 0$. In particular, $j\geq 22$. Notice that  we can write $n=50+55j$ in the following forms:
\begin{equation}\label{Eq:paraCaso5}
n=10\cdot (27+11r) +11(30u-10r), \, r\geq 0.
\end{equation} 

The sequence $\left\{x^{(5)}_r:r\geq 0\right\}=\left\{27+11r:r\geq 0\right\}=\left\{27,38,49,60,\ldots\right\}$ contains at least a prime number 
included in $\mathcal{A}$, namely  $x^{(5)}_{22}=269$. Then, if in the decompositions (\ref{Eq:paraCaso5}) we cannot find a value $r=1,\ldots, 22$ such that $\gcd(27+11r,30u-10r)=1$, at least we know that $n$ can be divided by $3, 7$ and $71$, as well as $269$, so $n=269\cdot t$, with $t\geq 21\cdot 71$. Thus, Proposition~\ref{P:formacionP} implies that $n\in \mathcal{A}$. In order to develop an appropriate reasoning, for $r=22$ it is necessary that in (\ref{Eq:paraCaso5}) either $30u -10\cdot 22\geq 2\cdot (269)+1$ if $\gcd(269,30u-10\cdot 22)=1$ or $30u - 10\cdot 22\geq 1$, otherwise; in both cases, it is sufficient to take $u\geq 26$. Since $u= 3+7\ell,$ this is achieved if we take $\ell\geq 4$. For the remaining cases, $\ell=0,1,2,3$, we have $n= 1260, 3570, 5880, 8190.$  Among them,  $3570, 5880, 8190 \in \mathcal{A}$. On the contrary, $1260\in\mathcal{NA}$, since $1260=10\cdot 5 + 11\cdot 110= 10\cdot 16 + 11\cdot 100=10\cdot 27 + 11\cdot 90=10\cdot 38 + 11\cdot 80=10\cdot 49 + 11\cdot 70=10\cdot 60 + 11\cdot 60=\cdots$.
This ends the case $b=5j$ in $\mathcal{C}_5$. According to our study, we emphasize that    
$$N_5:=\max\left\{\mathcal{NA}\cap \mathcal{C}_5\right\}= 1260.$$

\subsection{The class $\mathcal{C}_6$} We discuss the existence of values of $\mathcal{NA}$ in $\mathcal{C}_6=\left\{60 + 11k, \, k\geq 0\right\},$ and  try to find the maximum value in $\mathcal{NA}\cap \mathcal{C}_6.$ Let $n\in\mathcal{C}_6$, $n=60+11b,\, b\geq 0.$  It is immediate to check that the first values of $\mathcal{C}_6$, for $b=0,1,\ldots,12$ belong to $\mathcal{NA}$. Hence, we start our analysis with $b\geq 13$. Realize that, at the beginning, $a=6$.

$\bullet$ If, simultaneously, $b\neq\dot{2}$ and $b\neq \dot{3}$, then $60+11b\in \mathcal{A}$ whenever $b\geq 2a+1$, and this is guaranteed by our initial hypothesis on $b\geq 13$. 

$\bullet$  Assume that either $b=2j$ for $j\geq	7$, or $b=3j$ for $j\geq 5$. Now, it is necessary to consider another decomposition, say $n=10\cdot 17 + 11\cdot (b-10).$

-- If $b-10\neq \dot{17},$ we have $n\in \mathcal{A}$ once we check that the condition $b-10\geq 2\cdot 17 +1$ is satisfied. This holds for $b\geq 45$.  

(i) In this case, if in turn $b=2j$ for $j\geq 7$, then we have $n \in \mathcal{A}$ for $j\geq 23$; and, being $2j-10\neq\dot{17}$, the excluded cases in our analysis are $7\leq j\leq 21$. For them, the corresponding values of $n=10\cdot 17 + 11\cdot (2j-10)$ are $214, 236, 258, 280, 302, 324, 346, 368, 390, 412, 434, 456, 478, 500, 522.$ None of them is realized as an element of $\mathcal{A}$, all of them belong to $\mathcal{NA}$, (for instance,$522=10\cdot 6 + 11\cdot 42=10\cdot 17 + 11\cdot 32=\cdots$, here $32<2\cdot 17 +1$). 
Notice that, even, for $j=22$ we obtain $n=544\in\mathcal{NA}$.  

(ii) If, additionally, $b=3j$, with $j\geq 5$, then $n \in \mathcal{A}$ for $j\geq 15$ (notice $b\geq 45$); and imposing the restriction $3j-10\neq\dot{17}$, the excluded cases now are $j=5,6,7,8,10,11,12,13,14$, and its associate values $n=10\cdot 17 + 11\cdot (3j-10)$ are $225, 258, 291, 324, 390, 423, 456, 489, 522$ which are included in $\mathcal{NA}$. 

-- From now on, we consider that $b-10=\dot{17}$. We have two subcases: either $b=2j$, $j\geq 7$, or $b=3j$ with $j\geq 5$. In the first one, we have 
$2j-10=\dot{17}$ if and only if $j-5=\dot{17}$, therefore $j=22+17u, u\geq 0$; in the second case, we find $3j-10=\dot{17}$ if and only if 
$j=9+17u, u\geq 0$. Notice that  we can write $n=60+11b$ in the following forms:
\begin{equation}\label{Eq:paraCaso6}
n=10\cdot (6+11r) +11(b-10r), \, r\geq 0.
\end{equation} 

The sequence $\left\{x^{(6)}_r:r\geq 0\right\}=\left\{6+11r:r\geq 0\right\}=\left\{6,17,28,39,50, \ldots\right\}$ contains at least a prime number 
included in $\mathcal{A}$, namely  $x^{(6)}_{17}=193$. So, if in (\ref{Eq:paraCaso6}) we cannot find a value $r=1,\ldots, 17$ such that $\gcd(6+11r,b-10r)=1$, at least we know that $n$ can be divided by $2$ or $3$, by $17$ and $61$, as well as $193$, so $n=193\cdot t$, with $t\geq 34\cdot 61$. Then, Proposition~\ref{P:formacionP} implies that $n\in \mathcal{A}$. For $r=17$ it is necessary that in (\ref{Eq:paraCaso6}) either $b -10\cdot 17\geq 2\cdot (193)+1$ if $\gcd(193,b-10\cdot 17)=1$ or $b - 10\cdot 17\geq 1$, otherwise; in both cases, it is sufficient to take $b\geq 557$. 

According to the subcases above mentioned: (a) If $j=22+17u, u\geq 0$, then $b=2j=44+34u$,  and $b\geq 557$ for $u\geq 16$; consequently, $n=544+374u\in \mathcal{A}$. For the remaining cases $0\leq u\leq 15$ we obtain the following values of $n$, $ 544,$ $918,$ $1292,$ $1666,$ $2040,$ $2414,$ $2788,$ $3162,$ $3536,$ $3910,$ $4284,$ $4658,$ $5032,$ $5406,$ $5780,$ $6154.$ It is easy to check that $544, 918\in\mathcal{NA}$, and with the help of an algorithmic routine implemented in our personal computer, it can be checked that the rest of values can be decomposed in the form $10a'+11b'$, where the pairs $[a',b']$ are given respectively by 
\begin{eqnarray*}
& & [39,82], [39,116], [61,130], [39,184], [39,218], [61,232], [61,266],\\ 
& & [39, 320],  [61,334], [39,388], [39,422], [61,436], [39,490], [39,524].
\end{eqnarray*}
 (b) If $j=9+17u, u\geq 0$, then $b=3j=27+51u$,  and $b\geq 557$ for $u\geq 11$; in this case, 
$n=357+561u\in \mathcal{A}$. For $0\leq u\leq 10$ we obtain the following values of $n$, $357, 918, 1479, 2040, 2601, 3162, 3723, 4284, 4845, 5406, 5967;$ we find that $357, 918\in\mathcal{NA}$, whereas the rest of values belong to $\mathcal{A}$ and its respective pairs $[a', b']$ in the decomposition $10a'+11b'$ are given by $$ [28,109],[61,130],[28,211],[61,232],[28,313],  [61,334],[28,415],[61,436],[28,517].$$
This completes the study of $\mathcal{NA}\cap\mathcal{C}_6$. We have found that $$N_6:=\max\left\{\mathcal{NA}\cap \mathcal{C}_6\right\}= 918.$$ 

\subsection{The class $\mathcal{C}_7$} We discuss the existence of values of $\mathcal{NA}$ in $\mathcal{C}_7=\left\{70 + 11k, \, k\geq 0\right\},$ and try to find the maximum in $\mathcal{NA}\cap \mathcal{C}_7.$ Let $n=70+11b,\, b\geq 0.$
Firstly, we take $a=7$. It is necessary to consider $b\geq 15$ so $b\geq 2a+1$ is fulfilled. 
Thus, the first values of $\mathcal{C}_7$, for $b\leq14$ are $n=70,$ $81,$ $92,$ $103,$ $114,$ $125,$ $136,$ $147,$ $158,$ $169,$ $180,$ $191,$ $202,$ $213,$ $224,$ which are in $\mathcal{NA}$. Hence, we start our analysis with $b\geq 15$.

$\bullet$ If $b$ is not a multiple of $7$, then $\gcd(7,b)=1$ and $n=70+11b\in \mathcal{A}$.

$\bullet$ Assume that $b=7j$, with $j\geq 3$. We have to use the decomposition $n=10\cdot 18 + 11\cdot (7j-10)$, whenever $7j-10\geq 37,$ or $j\geq 7$. For the cases $j=3,\ldots,6$ we find the values $n=301, 378, 455, 532$, all of them 
being elements of $\mathcal{NA}$. From this point, we take $j\geq 7$. 

-- If $7j-10$, $j\geq 7$, is not multiple of $2$ nor $3$, we have that $n=180+11(7j-18)\in \mathcal{A}$ since $\gcd(18,7j-10)=1$ and 
 $7j-10\geq 37$. 

-- Instead, if $7j-10$ is either a multiple of $2$ or $3$, the condition on the greatest common divisor is not satisfied. Take into account that:
(a) if $7j-10=\dot{2}, j\geq 7$, then $j=8+2u, u\geq 0;$ (b) if $7j-10=\dot{3}, j\geq 7$, we find $j=7+3u, u\geq 0$. In the next step, we consider $n=10\cdot 29 + 11\cdot (7j-20)$, which provides numbers in $\mathcal{A}$ if $7j-20\geq 59,$ or $j\geq 12$. For $j=7,\ldots,11$ we only need to check the values $j=7,8,10$ due to (a) and (b), for which the corresponding values are $609, 686, 840$, which belong to $\mathcal{NA}$.

If, apart from the above requisites on $j$, we suppose that $7j-20$ is not a multiple of $29$, we obtain that $n=10\cdot 29 + 11\cdot (7j-20)\in \mathcal{A}$  for all $j\geq 12$. 
If, on the contrary, we assume that $7j-20=\dot{29}$, the above restrictions (a), (b) imply that either: (i) $7\cdot (8+2u)-20=\dot{29}$, that is, 
$18+7u=\dot{29}$, so $u=14+29\ell,\,  \ell\geq 0$; or (ii) $7(7+3u)-20=\dot{29}$, that is, $u=29\ell,\,\ell\geq 0.$ As usual, we write $n$ as follows:
\begin{equation}\label{Eq:paraCaso7}
n=10\cdot (18+11r) +11(7j-10-10r), \, r\geq 0.
\end{equation} 

The sequence $\left\{x^{(7)}_r:r\geq 0\right\}=\left\{18+11r:r\geq 0\right\}=\left\{18,29,40,51, \ldots\right\}$ exhibits prime numbers which are in $\mathcal{A}$, for instance  $x^{(7)}_{11}=139$. In this way, if in (\ref{Eq:paraCaso7}) we cannot find a value $r=1,\ldots, 11$ such that $\gcd(18+11r,7j-10-10r)=1$, at least we know that $n$ can be divided by $7$, $2$ or $3$, by $29$, by $73$, as well as $139$, so $n=139\cdot t$, with $t\geq 14\cdot 29$. Then, Proposition~\ref{P:formacionP} implies that $n\in \mathcal{A}$. In order to our reasoning makes sense, for $r=11$ we need that in (\ref{Eq:paraCaso7}) either $7j-10 -10\cdot 11\geq 2\cdot (139)+1$ if $\gcd(139,7j-10-10\cdot 11)=1$ or $7j-10 - 10\cdot 11\geq 1$, otherwise; in both cases, it is sufficient to take $j\geq 57$. Finally: (i) if $j=8+2u=36+58\ell$, $\ell\geq0$, it suffices to take $\ell\geq 1$ to ensure that the value $n$ is in $\mathcal{A}$; the unique exception appears for $\ell=0$ or $j=36$, and for this value we obtain $n=2842$, a new number in $\mathcal{A}$; (ii) if $j=7+3u=7+87\ell$, $\ell\geq 0$, we take again $\ell\geq 1$ to guarantee $j\geq 57$; in this case, the exception is given by $\ell=0$, or $j=7$, and for this value we obtain $n=609$, which is not in $\mathcal{A}$. 

This finishes the study of the set $\mathcal{NA}\cap \mathcal{C}_7$. We have obtained $$N_7:=\max\left\{\mathcal{NA}\cap \mathcal{C}_7\right\}= 840.$$ 

\subsection{The class $\mathcal{C}_8$} We pass to discuss the existence of values of $\mathcal{NA}$ in $\mathcal{C}_8=\left\{80 + 11k, \, k\geq 0\right\},$ as well as to try to find the maximum value in  
$\mathcal{NA}\cap \mathcal{C}_8.$ Let $n\in\mathcal{C}_8$, 
$n=80+11b,\, b\geq 0.$
To start, we take $a=8$. To get the necessary condition $b\geq 2a+1$, we must consider $b\geq 17$. 
Therefore, the first values of $\mathcal{C}_8$, for $b=0,1,\ldots,16$ are elements of $\mathcal{NA}$, namely $n=80,$ $91,$ $102,$ $113,$ $124,$ $135,$ $146,$ $157,$ $168,$ $179,$ $190,$ $201,$ $212,$ $223,$ $234,$ $245,$ $256.$ Hence, we start our analysis with $b\geq 17$.

$\bullet$ If $b$ is odd, then automatically $n=80+11b\in \mathcal{A}$ for all $b\geq 17$.  

$\bullet$ Assume that $b$ is even, $b=2j,\, j\geq 9$. Consider the decomposition
\begin{equation} \label{Eq:C8}
n=10\cdot 19 + 11\cdot (2j-10). 
\end{equation}
If we apply $b'\geq 2a'+1$, with $a'=19$ and $b'=2j-10$, it must be $j\geq 25.$ Thus, for the remaining values of, $9\leq j\leq 24$, we obtain that the respective values of $n$ are $278,$ $300,$ $322,$ $344,$ $366,$ $388,$ $410,$ $432,$ $454,$ $476,$ $498,$ $520,$ $542,$ $564,$ $586,$ $608,$ and all of them 
are in $\mathcal{NA}$. Consequently, from now on we assume that $b=2j$ with $j\geq 25$. From the inspection of (\ref{Eq:C8}), we deduce that if $2j-10$ is not a multiple of $19$, then we assure that $n\in \mathcal{A}$ for all $j\geq 25$. 
However, if $2j-10=\dot{19}$, that is, $j=24+19u, \, u\geq 1$, we have to continue with the process of looking for a suitable decomposition 
$10a'+11b'$. For instance, 
\begin{equation}\label{Eq:paraCaso8}
n=10\cdot (19+11r) +11\cdot (2j-10-10r), \, r\geq 0.
\end{equation} 

The sequence $\left\{x^{(8)}_r:r\geq 0\right\}=\left\{19+11r:r\geq 0\right\}=\left\{19,30,41,52, \ldots\right\}$ includes prime numbers which are in $\mathcal{A}$, for instance  $x^{(8)}_{8}=107$. In this way, if in the decompositions (\ref{Eq:paraCaso8}) we are not able to find a value $r=1,\ldots, 8$ such that $\gcd(19+11r,2j-10-10r)=1$, at least we know that $n$ can be divided by $2$, $19$,   as well as $107$, so $n=107\cdot t$, with $t\geq 38$. Then, by Proposition~\ref{P:formacionP} we know that $n\in \mathcal{A}$. In order to sustain our argument, for $r=8$ it is necessary that in (\ref{Eq:paraCaso8}) either $2j-10 -10\cdot 8\geq 2\cdot (107)+1$ if $\gcd(107,2j-10-10\cdot 8)=1$ or $2j-10 - 10\cdot 8\geq 1$, otherwise; in both cases, it suffices to take $j\geq 153$. Taking into account that $j=24+19u,$ $u\geq 1$, the inequality $j\geq 153$ is reached when $u\geq 7$, and then we can establish that $n=10\cdot 19 +11(2j-10)$ is in $\mathcal{A}$; the unique exceptions appear for $1\leq u\leq 6$, and for them 
we have $n=1026,1444,1862,2280,2698,3116,3534;$ in this list, 
$1026\in\mathcal{NA}$, since $1026=10\cdot 8+ 11\cdot 86=10\cdot 19+ 11\cdot 76=10\cdot 30+ 11\cdot 66=10\cdot 41+ 11\cdot 56=\cdots$, 
whereas the remaining numbers are elements of $\mathcal{A}$ as it can be easily verified. 
\begin{eqnarray*}
& & [41,56],[41,94],[41,132],[41,170],[41,208],\\
& & [63,226],[41,284],[41,322],[41,360].
\end{eqnarray*}  
Thus, we have finished the study of  case $b=2j$ corresponding to the set $\mathcal{NA}\cap \mathcal{C}_8$. Then $$N_8:=\max\left\{\mathcal{NA}\cap \mathcal{C}_8\right\}= 1026.$$ 

\subsection{The class $\mathcal{C}_9$} We consider the set  
$\mathcal{C}_9=\left\{90 + 11k, \, k\geq 0\right\},$ and try to guess the maximum value in $\mathcal{NA}\cap \mathcal{C}_9.$ Let $n\in\mathcal{C}_9$, $n=90+11b,\, b\geq 0.$
To start, we take $a=9$. To get the necessary condition $b\geq 2a+1$, we must consider $b\geq 19$. 
Therefore, the first values of $\mathcal{C}_9$, for $b=0,1,\ldots,18$ are elements of $\mathcal{NA}$, given by $ n = 90,$ $101,$ $112,$ $123,$ $134,$ $145,$ $156,$ $167,$ $178,$ $189,$ $200,$ $211,$ $222,$ $233,$ $244,$ $255,$ $266,$ $277,$ $288.$
Hence, we start our analysis with $b\geq 19$.

$\bullet$ If  $b \neq \dot{3}$, then directly $n=10\cdot 9 + 11\cdot b\in \mathcal{A}$ for all $b\geq 19$. 

$\bullet$ Assume that $b=3j$, with $j\geq 7$ and consider $n=10\cdot 31 + 11\cdot (3j-20).$

-- If, additionally, $3j-20$ is not a multiple of $31$, then $n$ is in $\mathcal{A}$ if $3j-20\geq 63$ or $j\geq 28$. So, we must analyze the exceptions $7\leq j\leq 28$. For these values, we find that $321,$ $354,$ $387,$ $420,$ $453,$ $486,$ $519,$ $552,$ $585,$ $618,$ $684,$ $750,$ $816,$ $882,$ $915,$ $948$ belong to $\mathcal{NA}$, and on the other hand $\{651,$ $717,$ $783,$ $849,$ $981,$ $1014\}\subset \mathcal{A}$, with associate pairs $[a',b']$ given respectively by $[20,47],[20,53],[20,59],[20,71],$ $[31,64].$
 
-- If $3j-20=\dot{31}$, $j\geq 7$, then $j=17+31u,\, u\geq 0.$ In this point, we consider the general formulation of the decompositions in $\mathcal{C}_9$, 
  \begin{equation}\label{Eq:paraCaso9}
n=10\cdot (9+11r) +11\cdot (3j-10r), \, r\geq 0.
\end{equation} 

The sequence $\left\{x^{(9)}_r:r\geq 0\right\}=\left\{9+11r:r\geq 0\right\}=\left\{9,20,31,42,53,\ldots\right\}$ includes prime numbers which are in $\mathcal{A}$, for instance  $x^{(9)}_{8}=97$. In this way, if in the decompositions (\ref{Eq:paraCaso9}) we are not able to find a value $r=1,\ldots, 8$ such that $\gcd(9+11r,3j-10r)=1$, at least we know that $n$ can be divided by $3$, $31$, as well as $97$, so $n=97\cdot t$, with $t\geq 93$. Then, Proposition~\ref{P:formacionP} gives $n\in \mathcal{A}$. For $r=8$ it is necessary that in (\ref{Eq:paraCaso9}) either $3j-10\cdot 8\geq 2\cdot (97)+1$ if $\gcd(97,3j-10\cdot 8)=1$ or $3j- 10\cdot 8\geq 1$, otherwise; in both cases, it suffices to take $j\geq 92$. Taking into account that $j=17+31u,$ $u\geq 0$, the inequality $j\geq 92$ is true when $u\geq 3$. In this case, the numbers $n$ will be in $\mathcal{A}$. For $u=0,1,2$, we have $n=651, 1674, 2697$, being $651$ and $2697$ in $\mathcal{A}$. Finally, $1674$ provides us the maximum of $\mathcal{NA}\cap\mathcal{C}_9$, (notice that $1674=10\cdot 9+11\cdot 144 = 10\cdot 20+11\cdot 134 =10\cdot 31+11\cdot 124 =10\cdot 42+11\cdot 114=10\cdot 53+11\cdot 104 =\cdots,$ and realize that $104<2\cdot 53 +1$. 
This concludes the part for a multiple of $3$, $b=3j$, and ends the case. We have found that $$N_9:=\max\left\{\mathcal{NA}\cap \mathcal{C}_9\right\}= 1674.$$ 

\subsection{The class $\mathcal{C}_{10}$} We consider the set $\mathcal{C}_{10}=\left\{100 + 11k, \, k\geq 0\right\},$ and try to find the maximum of $\mathcal{NA}\cap \mathcal{C}_{10}.$ Let $n\in\mathcal{C}_{10}$, $n=100+11b,\, b\geq 0.$
To start, we take $a=10$. To get the necessary condition $b\geq 2a+1$, we must consider $b\geq 21$. 
Thus, the first values of $\mathcal{C}_{10}$, for $b=0,\ldots,20$ are elements of $\mathcal{NA}$, given by $ n=100,$ $111,$ $122,$ $133,$ $144,$ $155,$ $166,$ $177,$ $188,$ $199,$ $210,$ $221,$ $232,$ $243,$ $254,$ $265,$ $276,$ $287,$ $298,$ $309,$ $320.$
We start our analysis with $b\geq 21$. Write $n=10\cdot 32 + 11\cdot (b-20).$

$\bullet$ If $b$ is odd, then $\gcd(32,b-20)=1$ and we can establish that $n\in \mathcal{A}$ if the condition $b'\geq 2a'+1$ 
is satisfied, where $a'=32, b'=b-20$. This supposes to take $b\geq 85$, $b$ odd. With the help of a computer, we find out whether the 
excluded values $b=21,\ldots, 83$ originate elements of $\mathcal{A}$ or elements of $\mathcal{NA}$. For these values we obtain $\{375, 485, 595, 705\}\subset \mathcal{NA}$ and the rest of the elements are in $\mathcal{A}$, whose respective pairs $[a',b']$ are given by
{\small{\begin{eqnarray*}
& [10,21], [10,23], [10,27], [10,29], [10,31], [10,33], [10,37], [10,39], [10,41], [10,43],\\
& [10,47], [10,49], [10,51], [10,53], [10,57],[10,59], [10,61], [10,63], [21,55], [10,67],\\
& [10,69],[10,71],[10,73],[21,65],[10,77],[10,79],[10,81],[10,83].
\end{eqnarray*}}}

$\bullet$ If $b$ is even, $b\geq 22$, or $b=2j$ with $j\geq 11$. Similarly to the odd case, we suppose that $b\geq 86$ and consider alone the exceptions $b=22+2\cdot s, 0\leq s\leq 31.$ From the corresponding values of $n=10\cdot 32 + 11\cdot (b-20)$, only $\{694, 716, 760, 782, 848, 892, 914, 958, 1024\}\subset \mathcal{A}$, with respective pairs $[a', b']$ given by 
$$[21,44], [21,46], [21,50], [21,52], [21,58], [21,62], [21,64], [21,68], [21,74],$$ and with the rest of the elements included in $\mathcal{NA}$ (by way of example, $1002=10\cdot 10 + 11\cdot 82= 10\cdot 21 + 11\cdot 72=10\cdot 32 + 11\cdot 62=10\cdot 43 + 11\cdot 52=\cdots$). 

From this point, we are restricted to the case $b=2j\geq 86$, or  $j\geq 43$. Next, we consider the combination $n=10\cdot 43 + 11\cdot (b-30).$
If $b-30=2j-30\neq\dot{43}$, then $\gcd(43,b-30)=1$ and we finish if we guarantee that $2j-30\geq 87$, or $j\geq 58$. This means 
that we have to study separately the values $43\leq j\leq 57$. For them, we find $\{1068, 1134, 1200, 1266, 1288, 1332\}\subset\mathcal{NA}$ (in this direction, let us show the decompositions of $1332$: $10\cdot 10 + 11\cdot 112=10\cdot 21 + 11\cdot 102= 10\cdot 32 + 11\cdot 92= 10\cdot 43 + 11\cdot 82= 10\cdot 54 + 11\cdot 72=\cdots$), 
and the rest are elements of $\mathcal{A}$ with the following respective pairs $[a', b']$:
\begin{eqnarray*}
[21,76], [21,80],[21,82], [21,86],[21,88], [21,92],[21,94],  [21,100], [21,104].
\end{eqnarray*}
 If $b-30=2j-30=\dot{43}$, then $j-15=\dot{43}$, so $j=58+43u,\, u\geq 0$. In order to assign a combination $10a'+11b'$ to $n$, we need to continue searching for new values $a', b'$. Then, we consider 
\begin{equation}\label{Eq:paraCaso10}
n=10\cdot (10+11r) +11\cdot (b-10r), \, r\geq 0.
\end{equation} 

The sequence $ \left\{x^{(10)}_r:r\geq 0\right\}=\left\{10+11r:r\geq 0\right\}$ contains several prime numbers which are in $\mathcal{A}$. We fix our attention in $x^{(10)}_{9}=109$. In this way, if in the decompositions (\ref{Eq:paraCaso10}) we are not able to find a value $r=1,\ldots, 9$ such that $\gcd(10+11r,2j-10r)=1$ for the number $n=10\cdot 10 +11\cdot(2j),$ with $j\geq 58$, $j=58+43u, u\geq0$, at least we know that $n$ can be divided by $2$ and $43$, as well as $109$, so $n=109\cdot t$, with $t\geq 86$. Then, Proposition~\ref{P:formacionP} gives $n\in \mathcal{A}$. In order to well guiding our reasoning, for $r=9$ it is necessary that in (\ref{Eq:paraCaso10}) either $2j-10\cdot 9\geq 2\cdot (109)+1$ if $\gcd(109,2j-10\cdot 9)=1$ or 
$2j- 10\cdot 9\geq 1$, otherwise; in both cases, it suffices to take $j\geq 155$. Bearing in mind that $j=58+43u,$ $u\geq 0$, the inequality $j\geq 155$ is valid when $u\geq 3$. In this case, the numbers $n$ will be in $\mathcal{A}$. For the remaining values of $u$, $u\in\{0,1,2\}$, we obtain $n\in\{1376, 2322, 3268\}$, which present the following respective decompositions $10\cdot 21+11\cdot 106,$ $10\cdot 65 +11\cdot 152$, and $10\cdot 21+11\cdot 278$, thus they
are in $\mathcal{A}$. This ends the discussion on the even numbers $b=2j$, and closes the case. We have found that $$N_{10}:=\max\left\{\mathcal{NA}\cap \mathcal{C}_{10}\right\}= 1332.$$

\subsection{The final bound. A table of the elements of $\mathcal{A}$} 
By collecting all the study developed for the sets $\mathcal{C}_m, 1\leq m\leq 10$, if we denote by $N_m$ the maximum value in 
$\mathcal{NA}\cap \mathcal{C}_m$, $1\leq m\leq 10$, we find:

\begin{tabular}{ccccc}
$N_1= 32,$ & $N_2=1560,$ & $N_3=1350,$ & $N_4=1140,$ &  $N_5=1260,$\\
$N_6=918,$ & $N_7=840,$ & $N_8=1026,$ & $N_9=1674,$ & $N_{10}=1332.$
\end{tabular}

On the other hand, the maximum value of $\mathcal{NA}$ being a multiple of $11$ is $N_{11}=1320$ as Proposition~\ref{P:para11} and Corollary~\ref{C:cota11} show. Therefore, $$M=\max\{\mathcal{NA}\}=\max\{N_m: 1\leq m \leq 11 \}=\boxed{1674}.$$ Once we know the maximum of $\mathcal{NA}$, with the help of a mathematical software and a few of patient we obtain all the elements in $\mathcal{A}$, which we gather in the following table:

\begin{center} \label{table}
\resizebox{12cm}{!} {
\begin{tabular}{c|c}
\hline Intervals & Numbers in $\mathcal{A}$ \\ \hline
 $n\in[1,100]$ & $1,8,11,43,54, 65, 75, 76, 87, 97, 98$ \\ \hline 
 $n\in[101,200]$ & $107, 109, 118, 119, 120, 131, 139, 140, 141, 142, 151, 153,$  \\ 
   $ $   & $161, 163, 164, 171, 173, 175, 182, 183, 184, 185, 186, 193, 197$  \\
 \hline
$n\in[201,300]$ &  $203, 204, 205, 206, 207, 208, 217, 219, 226, 227, 229, 230, 235, 237, 239,$\\
$ $ & $241, 246, 247, 248, 249, 250, 251, 252, 257, 259, 263, 267, 268, 269, 271,$ \\
$ $ & $ 272, 273, 274, 279, 281, 283, 285, 289, 290, 292, 293, 295, 296, 299$\\
\hline
$n\in[301,400]$ & $303, 305, 307, 311, 312, 313, 314, 315, 316, 317, 318, 323, 329, 331, 332, $\\
 $ $ & $333, 334, 335, 336, 337, 338, 339, 340, 343, 345, 347, 349, 351, 353, 355, $\\
 $ $ & $356, 358, 359, 361, 362, 363, 365, 367, 369, 371, 373, 374, 376, 377, 379, $\\
 $ $ & $381, 382, 383, 384, 385, 389, 391, 395, 396, 397, 398, 399, 400$\\
\hline 
$n\in[401,500]$ & $[401,500] \setminus \{408, 410, 412, 414, 416, 420, 423, 426, 430, $\\
 $ $ & $432, 434, 435, 436, 452, 453, 454, 455, 456, 458, 473, $\\
 $ $ & $474, 476, 478, 480, 485, 486, 490, 492, 496, 498, 500 \}$\\
\hline 
$n\in[501,600]$ & $[501,600] \setminus\{518,519,520,522,525,532,540,542,544,546,$ \\
 & $552,558,562,564,584,585,586,590,594,595,600 \}$ \\
\hline 
$n\in[601,700]$ & $[601,700]\setminus\{606,608,609,612,618,624,$ \\ & $628,650,672,678,684,686,690,700\}$\\
\hline 
$n\in[701,800]$ & $[701,800]\setminus\{702, 705, 710, 738, 744, 750, 756\}$\\
\hline 
$n\in[801,900]$ & $[801,900]\setminus\{804, 810, 820, 826, 836, 840, 870, 876, 882\}$\\
\hline 
$n\in[901,1000]$ & $[901,1000]\setminus\{915, 918, 920, 930, 936, 942, 948, 980, 988\}$\\
\hline 
$n\in[1001,1100]$ & $[1001,1100]\setminus\{1002, 1008, 1020, 1026, 1030, 1068, 1074\}$\\
\hline 
$n\in[1101,1300]$ & $[1101,1300]\setminus\{1134, 1140, 1200,1260, 1266, 1274, 1288\}$\\
\hline 
$n\in[1301,1700]$ & $[1301,1700]\setminus\{1320, 1332, 1350,1560, 1674\}$\\
\hline 
$n > 1674$ & All the values\\ \hline
\end{tabular}
}
\end{center}

\begin{remark}
After Table \ref{table} was done by performing a routine program in a standard mathematical software, Professor Pedro A. García Sánchez commented us the existence of an algebraic platform of free access, GAP (see \cite{GAP4}, a system for computational discrete algebra), and the corresponding package for calculating the numbers of $\mathcal{A}$ (see \cite{Delgado}); for instance, the reader interested in algorithm procedures can compute them by implementing the following instructions: 
\end{remark}
\vspace{-2mm}
{\small{
\noindent \texttt{s:=NumericalSemigroup(10,11);}
\newline \texttt{belong:=x->ForAny(Factorizations(x,s),p->p[2]>=2*p[1]+1 and Gcd(p)=1);} 
\newline \texttt{Filtered(Intersection(s,[0..2000]),belong);}}} $\hfill\Box$

\section{Conclusion}

To sum up we have proved that the set $$\mathcal{A} = \left\{10\cdot a + 11\cdot b \ | \gcd(a,b)=1, a\geq 1, b\geq 2a+1 \right\}$$ \noindent is unbounded. Moreover, we have found the greatest number that does not belong to $\mathcal{A}$, that is, $M= 1674$ and we have determined in detail all the element of $\mathcal{A}$, which are collected in Table \ref{table}.

As far as we know, this problem has not been previously studied in the literature. Observe that our result is a particular case from the analysis of a certain difference equation of order $k=4$, Equation (\ref{Eq:G4}). Obviously, such equation can be extended to an arbitrary order $k \geq 5$. In \cite{Linero}, the authors propose a conjecture establishing a relationship between the general set of periods and the semigroup generated by numbers $3k-1$ and $3k-2$; the exact knowledge of such set of periods perhaps would need some additional conditions to the elements of the semigroup. In this sense, it is an open problem to determine the extra conditions that should be considered and, consequently, to determine the gaps which those conditions generate in the semigroup and to analyse the boundedness character of these gaps. To this regard, the authors are aware of the existence of certain additional conditions, to wit, $1\geq a \geq 10$, that yield to the unboundedness character of the complementary, since it can be proved that there exists a strictly increasing sequence of positive numbers $(n_j)_j$ such that $2^{n_j} \notin \mathcal{A}$. 

Therefore, some general interesting questions for algebrists would be: the study of the boundedness character of the complementary of semigroups generated by two coprimes numbers, $p$ and $q$, whose elements are originated by combinations $a\cdot p + b \cdot q$ being $a,b $ natural integers satisfying $\gcd(a,b)=1$, as well as to find in the bounded case a formula for the biggest positive integer not representable with the assumptions of being coprimes $a,b$; to extend the previous problem to the case where $p,q$ are not necessarily coprime; to analyse the problem of adding extra conditions to the numbers $a,b$ apart from the fact of being coprime, as for example, linear inequalities. 

\section*{Acknowledgements}

We sincerely thank Prof. Pedro A. García Sánchez from University of Granada, Spain, and Prof. Christopher O'Neill from San Diego Statal University, CA, USA, for providing us some useful comments and to inform us about the existence of GAP system. 

This paper has been partially supported by Grant MTM2017-84079-P (AEI/FEDER,UE), Ministerio de Ciencias, Innovación y Universidades, Spain.

\end{document}